\documentclass[11pt, a4paper]{amsart}
\usepackage[utf8]{inputenc}
\usepackage[T1]{fontenc}
\usepackage[margin=0.9in]{geometry}

\usepackage{comment}
\usepackage{lmodern}
\usepackage{upgreek}
\usepackage{color}
\usepackage{amsmath}
\usepackage{amssymb}
\usepackage{amsthm}

\usepackage{pdfpages}

\usepackage{fancyhdr}

\usepackage{setspace}

\newcommand{\Q}{\mathbb{Q}}

\newcommand{\C}{\mathbb{C}}

\newcommand{\F }{\mathbb{F} }

\newcommand{\A}{\mathbb{A}}
\newcommand{\G}{\mathbb{G}}

\usepackage{hyperref}
\hypersetup{
	colorlinks=true,
	linkcolor=blue,
	urlcolor=blue,
	citecolor=blue,
	pdftitle={Products of two irreducibles in residue classes over \(\mathbb{F}_q[t]\)},
	pdfauthor={Likun Xie}
}

\numberwithin{equation}{section}
\newtheorem{theorem}{Theorem}[section]

\newtheorem{proposition}[theorem]{Proposition}
\newtheorem{corollary}[theorem]{Corollary}

\theoremstyle{remark}
\newtheorem{remark}[theorem]{Remark}
\newtheorem{defi}{Definition}[section]

\theoremstyle{definition}

\newtheorem*{prob*}{Problem}

\title{On a Conjecture of Erd\H{o}s over Function Fields}

\author{Likun Xie}
\address{Max-Planck-Institut für Mathematik
	Vivatsgasse 7, 53111, Bonn, Germany} 
\email{xie@mpim-bonn.mpg.de}
\subjclass[2020]{Primary 11T55, 11T24}
\keywords{Erdős conjecture, function fields, trace functions}

\begin{document}

	\begin{abstract}
		
Using Katz's equidistribution framework, we show that for any squarefree polynomial
$f \in \mathbb{F}_q[t]$ of degree $n \ge 2$, every residue class modulo $f$ can be represented
as a product of two monic irreducible polynomials of degree at most $n$, provided $q$ is
sufficiently large in terms of $n$. This gives the function-field analogue of a conjecture
of Erd\H{o}s in the large-$q$ regime.

Sawin~\cite{sawin} previously proved this representation with stronger square-root cancellation
via a higher-dimensional sheaf-theoretic construction. This note presents a one-dimensional
argument that yields a natural $q^{-1/2}$ saving.
	\end{abstract}

	\maketitle

	\section{Introduction}
	Many central conjectures in additive and multiplicative number theory remain unresolved over the integers. A notable example is the twin prime conjecture, which asserts the infinitude of prime pairs differing by two. The strongest progress toward this conjecture, achieved through sieve methods, is due to Chen~\cite{chen}, who proved that there exist infinitely many primes~$p$ such that $p+2$ is the product of at most two primes. By contrast, in the setting of global function fields, the availability of geometric techniques has led to major advances. In particular, Sawin and Shusterman~\cite{twin} resolved the function field analogues of both Chowla's  and the twin prime conjecture.

	A classical conjecture of Erd\H{o}s~\cite{erdos} asserts that for every sufficiently large prime modulus $q$, each reduced residue class modulo $q$ can be represented as a product $p_1p_2$ of two primes with $p_1,p_2\le q$. Building on   multiplicative dense–model and transference framework, 
	Matom\"aki and Ter\"av\"ainen~\cite{matomaki} established a ternary version of this conjecture (representation by three primes) and proved the binary case for all but a small exceptional set of residue classes.

	Let $q$ be a prime power and $k=\F_q$ the finite field of order $q$.  
	Let $f \in k[X]$ be a squarefree polynomial of degree $n \geq 2$, and set
	\[
	B := k[X]/(f(X)).
	\]
	Denote by $B^\times$ the multiplicative group of $B$, and let 
	$\chi : B^\times \to \C^\times$ be a multiplicative character.  
	Define
	\[
	E( k,f) := \Bigl\{\, a \in B^\times : a \equiv f_1 f_2 \pmod{f},\ 
	f_1,f_2 \in k[X] \ \text{monic irreducible with } \deg(f_i)\leq n \Bigr\}.
	\]
	
Lemma~9.14 of Sawin~\cite{sawin} yields a function field analogue of Erd\H{o}s' conjecture: if \( q > 1024 e^2 \) and \( n \) is sufficiently large, then every residue class modulo any squarefree polynomial of degree \( n \) can be written as a product of two monic irreducible polynomials of degree \( n \). Here the natural size parameter in the  function field analogue of Erd\H{o}s' conjecture is
\[
|f| = q^{\deg f} = q^n,
\]
which is taken to be  sufficiently large.
 
Specializing Lemma~9.14 of Sawin~\cite{sawin} with the substitutions
\[
c \mapsto 0, \quad n \mapsto 2n, \quad m \mapsto n, \quad
g \mapsto f, \quad
\omega \mapsto 2, \quad n_1 \mapsto n,\; n_2 \mapsto n, \quad h \mapsto 1,
\]
the left-hand side of the lemma counts the number of pairs of monic irreducible polynomials \( (f_1, f_2) \) of degree \( n \) such that
\[
f_1 f_2 \equiv a \pmod{f},
\]
minus a main term of size \( q^{\,n} \). The error term appearing in Lemma~9.14 is
\[
O\!\left( \frac{(2n)!}{(n!)^2} (8 e \sqrt{q})^{n} \right)
= O\!\left( (32 e \sqrt{q})^n \right),
\]
which is smaller than the main term whenever  \( q > 1024 e^2 \) and $n$ sufficiently large. Therefore \( E(k,f)=B^\times \) holds for all sufficiently large \( n \).

Here we give a one-dimensional construction to establish the function field analogue of Erd\H{o}s' conjecture in the large-\( q \) regime, using Katz's convolution--equidistribution framework~\cite{katz_book, katz}.  The method here yields an error term of size \( O(n^{\,n} q^{\,n - 1/2}) \) with the implied constant given explicitly.  It is not as strong as the square-root cancellation obtained in~\cite{sawin}, but it is sufficient for our purposes when $q$ is large relative to $n$.

In our approach, the size of the saving is controlled by the dimension of the
support on which Deligne's theorem is applied. We start with a weight-0
middle-extension sheaf $F$ on the curve $\A^1_k[1/f]$ arising from the relevant
$L$-function, push it forward along a map $p:\A^1_k[1/f]\to G$ (either \( \G_m \) or the torus \( T = \operatorname{Res}_{B/k}\G_m \) may be used), and obtain a perverse object $M$ on $G$
normalized to weight~$0$, attached to a fixed representation. For each residue
class $a\in G(k)$, the main term in the twisted character sum is
\[
q^{n}\,\operatorname{Tr}\!\left(\operatorname{Frob}_{k,a}\mid M\right).
\]
Since $M$ is non-punctual and its support is a curve, Deligne's Riemann
Hypothesis yields the uniform pointwise bound
\[
\bigl|\operatorname{Tr}(\operatorname{Frob}_{k,a}\mid M)\bigr| \;\ll\; q^{-1/2}.
\]

By contrast, in Sawin's setting~\cite{sawin} one starts from the higher-dimensional variety $X=(\A^1_k[1/f])^{2n}$ with its natural $S_{2n}$-action and pushes forward along $\pi:X\to T$ (this factors through the $2n$-th symmetric power of $\A^1_k[1/f]$). For a residue class  $a\in T(k)$, the error term is related to the Frobenius trace on \(R\Gamma_c(X_a,N)\;\cong\;(R\pi_!N)_a, \) where $N$ is an $S_{2n}$-equivariant coefficient sheaf constructed from a
representation $\rho$ of $S_{2n}$.
 A key input is that $R\Gamma_c(X_a,N)$ concentrates near the middle degrees (namely, in degrees \(n\) and \(n+1\)). Applying Deligne's Riemann Hypothesis together with  representation-theoretic combinatorics yields square-root cancellation, giving a $q^{-n/2}$ saving
relative to the main term of size $q^n$ in this product-of-two-irreducibles problem.
This leads to significantly stronger uniform bounds.

Here we give a simpler one-dimensional argument in the large-\(q\) regime 
for the function field analogue of Erd\H{o}s' conjecture.  Beyond this 
specific application, the same template applies to sums of the form
\[
\sum_{\chi \in \operatorname{TotRamGen}(k,f)} 
\chi(c)\,\operatorname{Tr}\!\left(\lambda(\theta_{k,f,\chi})\right),
\]
for any \(c \in B^{\times}\) and any fixed virtual representation \(\lambda\) of 
the relevant monodromy group, yielding a \(q^{-1/2}\) saving with constants 
depending only on \(\lambda\) and \(n\) (see Proposition~3.4 and 
Corollary~3.5).  For example, replacing \(\Lambda(f_1)\Lambda(f_2)\) by \(\mu(f_1)\mu(f_2)\) or by \(d_r(f_1) d_s(f_2)\) corresponds
to replacing
\(\operatorname{Tr}(\theta_{\chi}^{\,i})\,\operatorname{Tr}(\theta_{\chi}^{\,j})\)
by
\(\operatorname{Tr}(\wedge^{i}\theta_{\chi})\,\operatorname{Tr}(\wedge^{j}\theta_{\chi})\)
and
\(\operatorname{Tr}(\operatorname{Sym}^{i}(\theta_{\chi}^{\oplus r}))\,
\operatorname{Tr}(\operatorname{Sym}^{j}(\theta_{\chi}^{\oplus s}))\),
respectively.  In both cases, the same one-dimensional argument yields a
\(q^{-1/2}\) saving.

We now state our main result.
	\begin{theorem}\label{thm:main}
	Let \(n \ge 2\). Then there exists a  constant \( Q(n) > 0 \) such that 
		for every finite field \( k = \F_q \) with \( q \ge Q(n) \) and every squarefree polynomial \( f \in k[X] \) of degree~\( n \), we have
		\[
		E(k,f) = B^\times.
		\]
		Moreover, one may take \( Q(n) \) to be of the form \( n^{\,\kappa n} \) for some absolute constant \( \kappa > 0 \);
		for instance, \( \kappa = 23 \) suffices.
	\end{theorem}

	\section{Initial set-up}

	Let $q$ be a prime power and $k=\F_q$ the finite field of order $q$, with fixed algebraic closure $\bar{k}$.  
	Let $f \in k[X]$ be a squarefree polynomial of degree $n \geq 2$, and set
	\[
	B := k[X]/(f(X)),
	\]
	which is a finite \'etale $k$-algebra of degree $n$. Denote by $B^\times$ the multiplicative group of $B$. 
	Fix $d\geq 1 $. For $a \in B^\times$, define
	\begin{equation} 
		S(a;d) := \frac{1}{|B^\times|} \sum_{\chi \bmod f} \chi(a^{-1})
		\left( \sum_{\substack{\deg g \leq d\\ g \text{ monic}}} \chi(g)\,\Lambda(g) \right)^{\!2}, \nonumber
	\end{equation}
	where  
	$\Lambda$ denotes the von Mangoldt function for $\F_q[t]$, given by
	\[
	\Lambda(h) = 
	\begin{cases}
		\deg p, & \text{if } h = u p^k \text{ for some monic irreducible } p \in \F_q[t],\ u \in \F_q^\times, \ k \geq 1, \\
		0, & \text{otherwise}.
	\end{cases}
	\]
	
	Note that one may   restrict to $\deg g = d$  since the contribution is dominated by polynomials of the highest degree.

	By the explicit formula for the prime polynomial theorem  \cite[Proposition~2.1]{rosen}, we have
	\[
	\sum_{\substack{g \ \text{monic} \\ \deg g = d}} \Lambda(g) \;=\; q^d.
	\]

	The contribution from the principal character $\chi=\chi_0$ is
	\begin{align}\label{def_M}
		M(a,d) 
		:= \frac{1}{|B^\times|}
		\left( \sum_{\substack{\deg g \leq d \\ (g,f)=1 \\ g \ \text{monic}}} \Lambda(g) \right)^{\!2}   
		&= \frac{1}{|B^\times|}
		\left(\sum_{k\leq d}\Biggl(\;
		\sum_{\substack{\deg g = k \\ g \ \text{monic}}}\Lambda(g)
		\;-\!\!\!\sum_{\substack{p \mid f \\ \deg p \mid k \\ p \ \text{monic irreducible}}}\deg p
		\Biggr)\right)^{\!2} \nonumber \\[6pt]
		&= \frac{1}{|B^\times|}
		\left(\sum_{k\leq d}\big(q^k+O(\deg f)\big)\right)^{\!2} \nonumber \\[6pt]
		&= \frac{1}{|B^\times|}
		\left(\frac{q^{\,d+1}-q}{q-1}+O(dn)\right)^{\!2} \nonumber \\[6pt]
		&= q^{\,2d-n}\!\left(1+O\!\left(\tfrac{(dn)^2}{q}\right)\right).
	\end{align}
	More precisely, for $q\geq n$, we have
	\begin{equation}\label{lower_bound_M}
		M(a,d)
		\ge
		\frac{1}{|B^\times|}
		\left(\sum_{\substack{k \le d}} (q^k - n)\right)^{\!2}   \ge
		\frac{1}{q^n}\,
		\bigl(q^d - dn\bigr)^{\!2}  =
		q^{\,2d - n}
		- 2dn\,q^{\,d - n}
		+ \frac{(dn)^2}{q^n}.  
	\end{equation}
	To establish Theorem~\ref{thm:main}, it suffices to show that, when $d = n$, 
uniformly for all $a \in B^\times$, we have
\[
R(a,n)
\;=\;
o_n\!\left(M(a,n)\right).
\]
Using the orthogonality of characters and  Weil's bound \eqref{weil}, we have
\[
\frac{1}{|B^\times|}\sum_{a\in B^\times}\left|R(a;d)\right|^2
\;=\;\frac{1}{|B^\times|^2}
\sum_{\substack{\chi \bmod f \\ \chi\ne \chi_0}}
\Biggl|\sum_{\substack{\deg g \leq d\\ g \text{ monic}}}\chi(g)\,\Lambda(g)\Biggr|^{4}
\;\ll\; \frac{1}{q^{2n}} \cdot q^n \cdot (n q^{d/2})^4
\;\leq \; n^4 q^{\,2d-n}.
\]
Define the set of “bad” elements
\[
\mathcal{E} \;:=\; \{\,a \in B^\times : S(a;d) \le 0 \,\}.
\]
If $S(a;d) \le 0 $, then 
\[
|R(a;d)| \;\ge\; M(a,d).
\]

Applying Markov's inequality, we obtain
\begin{equation}\label{exceptional set}
	|\mathcal{E}|
	\;\le\; |B^\times|\cdot \frac{\mathbb{E}_a\,|R(a;d)|^2}{M(a,d)^2}
	\;\ll\; |B^\times|\cdot \frac{n^4 q^{\,2d-n}}{q^{\,4d-2n}}
	\;=\; n^4\,q^{\,2n-2d}.
\end{equation}
Therefore, for $d>n$ the exceptional set $\mathcal{E}$ is empty once $q$ is sufficiently large.   
To prove Theorem~\ref{thm:main}, one needs  to treat the critical case $d=n$, where the above argument yields only the a priori bound $|\mathcal{E}| \ll n^{4}$.

A direct application of Weil's bound~\eqref{weil} shows that
\[R(a,n)
= \frac{1}{|B^\times|} 
\sum_{\substack{\chi \bmod f \\ \chi \neq \chi_0}} 
\chi(a^{-1}) 
\left( \sum_{\substack{\deg g \le n \\ g \text{ monic}}} \chi(g)\,\Lambda(g) \right)^{\!2}
\;\le\; n^2 q^n , \]
which is insufficient for the desired result.
We will  restrict to a generic subset of characters for which a \(q^{-1/2}\)-saving relative to the main term \(M(a,n)\) can be obtained.

	Let 
	$\chi : B^\times \to \C^\times$ be a multiplicative character.  
	We extend $\chi$ to all monic polynomials $h \in k[X]$ by setting 
	\[
	\chi(h) := 
	\begin{cases}
		\chi(\bar{h}), & \text{if } \bar{h}\in B^\times, \\[4pt]
		0, & \text{otherwise},
	\end{cases}
	\]
	where $\bar{h}$ is the image of $h$ in $B$.
	
	The Dirichlet $L$-function attached to $\chi$ is the power series in $\C[[T]]$ defined by
	\[
	L(\chi,T) := \sum_{\substack{g \in k[X] \\ g \ \text{monic}}} \chi(g)\, T^{\deg g}.
	\]
	If $\chi \neq \chi_0$, then $L(\chi,T)$ is a polynomial in $T$ of degree $n-1$.  
	We may factor it as
	\[
	L(\chi,T) = \prod_{i=1}^{\,n-1} \left(1 - \beta_i T\right).
	\]
	By the Riemann Hypothesis for function fields, proved by Weil~\cite{weil}, each reciprocal root $\beta_i$ either equals~$1$ or satisfies 
	\[
	|\beta_i| = q^{1/2}.
	\]
	Define
	\[
	\Psi(m,\chi) := \sum_{\substack{f \ \text{monic} \\ \deg f = m}} \Lambda(f)\,\chi(f).
	\]
	Then by Weil's theorem we have
	\begin{equation}\label{weil}
		\left|\Psi(m,\chi)\right| \;\leq\; (n-1)\,q^{m/2},
		\qquad \chi \neq \chi_0.
	\end{equation}	
	Write $f$ as a product of distinct monic irreducible polynomials,
	\[
	f = \prod_i f_i.
	\]
	Note that every character $\chi$ of $B^\times$ decomposes as a product
	\begin{equation}\label{character_component}
		\chi = \prod_i \chi_i,
	\end{equation}
	where $\chi_i$ is a character of $B_i^\times$ with $B_i = k[X]/(f_i(X))$.
	
	We say that $\chi$ is \emph{primitive} if each character component $\chi_i$ is nontrivial.  
	We say that $\chi$ is an \emph{odd} character if its restriction to the subgroup
	$k^\times \subset B^\times$ is nontrivial.  
	
	We now adopt the terminology from~\cite{katz}.
	
	\begin{defi}\label{def_TotRam}
		Let $k$ be a finite field and $f \in k[X]$ a squarefree polynomial, and set
		\(B := k[X]/(f(X)).\)
		We say that  a character 
		\(	\chi : B^\times \to \C^\times \)
 is \emph{totally ramified} if it is both primitive and odd.  
		We denote by $\operatorname{TotRam}(k,f)$ the set of all totally ramified characters of $B^\times$.
	\end{defi}

	For a totally ramified character $\chi$, each reciprocal root $\beta_i$ of $L(\chi,T)$ satisfies 
	$|\beta_i| = q^{1/2}$.  
	The unitarized $L$-function
	\(L \left(\chi,\ {T}/{q^{1/2}}\right)\)
	is the reversed characteristic polynomial
	\(\det \left(1 - T \Theta_\chi\right),\)
	where 
	\[
	\Theta_\chi = \operatorname{diag}\!\left(\tfrac{\beta_1}{q^{1/2}}, \dots, \tfrac{\beta_{n-1}}{q^{1/2}}\right) 
	\in U(n-1).
	\]
	
	Since conjugacy classes in $U(n-1)$ are determined by their characteristic polynomials, 
	there exists a well-defined conjugacy class
	\begin{equation}\label{def_conjugacy_class}
		\theta_{k,f,\chi} \;\in\; U(n-1)
	\end{equation}
	such that
	\[
	L\!\left(\chi,\ {T}/{q^{1/2}}\right) = \det \left(1 - T \theta_{k,f,\chi}\right).
	\]
	Moreover, we have the explicit formula
	\[
	\Psi(m,\chi) = q^{m/2}\, \operatorname{tr}\!\left(\theta_{k,f,\chi}^m\right).
	\]
	Let~$\operatorname{All}(k,f)$ denote the set of all characters on~$B^\times$.  
By~\cite[Lemma~6.4]{katz}, we have
	\[
	\left|\operatorname{All}(k,f)\right| - \left|\operatorname{TotRam}(k,f)\right|
	\;\le\;
	q^n - 1 - (q - 2)^n + \sum_{0 \le i \le n-1} q^i
	\;\le\;
	2(n{+}1)\, q^{\,n-1}.
	\]

	Applying~\eqref{weil},  the contribution from characters 
	$\chi \ne \chi_0$ with 
	$\chi \notin \operatorname{TotRam}(k,f)$ 
	is bounded by
	\[
	\frac{1}{|B^\times|}\, 2(n{+}1)\, q^{\,n-1}
	\sum_{\substack{i \le n \\ j \le n}}
	n^2 q^{\tfrac{i+j}{2}}
	\;\;\le\;\;
	2(n{+}1)\, n^4\, q^{\,n-1}.
	\]
	Consequently, we may write
	\begin{align}\label{R(a,n)}
		R(a,n)
		&=\;
		\frac{1}{|B^\times|}
		\sum_{\chi \in \operatorname{TotRam}(k,f)} 
		\chi(a^{-1})
		\sum_{i,j \le n} 
		q^{\tfrac{i+j}{2}}\,
		\operatorname{Tr}\!\left(\theta_{k,f,\chi}^i\right)
		\operatorname{Tr}\!\left(\theta_{k,f,\chi}^j\right)
		\nonumber \\[6pt]
		&\quad+\;
		\frac{1}{|B^\times|} 
		\sum_{\substack{\chi \notin \operatorname{TotRam}(k,f)\\[2pt]\chi\ne \chi_0}} 
		\chi(a^{-1}) 
		\left(
		\sum_{m\leq n} \Psi(m, \chi)
		\right)^{\!2}
		\nonumber \\[6pt]
		&\le\;
		\frac{1}{|B^\times|} \left|
		\sum_{\chi \in \operatorname{TotRam}(k,f)} 
		\chi(a^{-1})
		\sum_{i,j \le n} 
		q^{\tfrac{i+j}{2}}\,
		\operatorname{Tr}\!\left(\theta_{k,f,\chi}^i\right)
		\operatorname{Tr}\!\left(\theta_{k,f,\chi}^j\right)\right|
		\;+\;
		2(n{+}1)\, n^4\, q^{\,n-1}.
	\end{align}

\section[\texorpdfstring{Bound for \(R(a,n)\)}{Bound for R(a,n)}]{Bound for \( R(a,n) \)}
	In this section, we continue to work with a squarefree monic polynomial 
	$f \in k[X]$ of degree $n \ge 2$,  
	and set $B := k[X]/(f(X))$.  
	Let $\chi$ be a character of $B^\times$.  
	Choose a prime~$\ell$ that is invertible in~$k$, and fix an embedding
	\[
	\overline{\Q} \hookrightarrow \overline{\Q}_\ell.
	\]
	In this way, we regard $\chi$ as a $\overline{\Q}_\ell^\times$-valued character of~$B^\times$.
	
	We begin by recalling the construction of the perverse sheaf 
	$N(\lambda,\chi)$  following~\cite{katz}.  
	Before doing so, we briefly review the key setup underlying this construction.

	Define the group-valued functor $\mathbb{B}^\times$ on the category of $k$-algebras by
	\[
	\mathbb{B}^\times(R) := (B \otimes_k R)^\times 
	= (R[X]/(f(X)))^\times.
	\]
	This functor is represented by  the Weil restriction
	\(\operatorname{Res}_{B/k}(\mathbb{G}_{m,B}),\)
	which exists by~\cite[Thm.~7.6.4]{neron} and is a
	smooth commutative group scheme over~\(k\) by~\cite[Prop.~7.6.5]{neron}.
	
	We have an embedding
	\[
	i_f : \mathbb{A}^1_k[1/f] \longrightarrow \mathbb{B}^\times
	\]
	defined on $R$-valued points by
	\[
	i_f(R) : \mathbb{A}^1_k[1/f](R)
	= \{\, t \in R \mid f(t) \in R^\times \,\} 
	\longrightarrow \mathbb{B}^\times(R),
	\qquad
	t \longmapsto X - t.
	\]
	There is also a morphism of algebraic groups
	\[
	\operatorname{Norm}_{B/k} : \mathbb{B}^\times \longrightarrow \mathbb{G}_m
	\]
	given on $R$-valued points by the norm maps \(\operatorname{Norm}_{\mathbb{B}(R)/R}\)
	restricted to units.  
	These morphisms satisfy the relation
	\begin{equation}\label{norm_relation}
		\operatorname{Norm}_{B/k} \circ i_f = (-1)^{\deg f} \cdot f,
	\end{equation}
	see~\cite[Lemma~3.1]{katz}.

	For a character $\chi:B^\times\to\overline{\mathbb{Q}}_\ell^\times$, let $\mathcal L_\chi$ denote the associated rank-$1$ character sheaf on $\mathbb{B}^\times$.
	The sheaf $\mathcal{L}_\chi$ is lisse of rank~$1$ and pure of weight~$0$.  

Let \(G\) be a connected commutative algebraic group over \(k\).   
Write \(D_c^b(G,\overline{\mathbb{Q}}_\ell)\) for the bounded constructible derived
category, and let \(\mathrm{Perv}(G)\) denote its full subcategory of perverse sheaves.  
For \(K \in D_c^b(G,\overline{\mathbb{Q}}_\ell)\) and \(n \ge 1\), let \(k_n\) be the degree-\(n\) extension of \(k\).  
	The \emph{trace function} associated to~$K$ is the map
	\[
	\operatorname{Tr}(\operatorname{Frob}_{k_n,-} \mid K)
	\;:\;
	G(k_n)
	\;\longrightarrow\;
	\mathbb{C}
	\]
	defined by
	\[
	\operatorname{Tr}(\operatorname{Frob}_{k_n,x} \mid K)
	\;=\;
	\sum_{i \in \mathbb{Z}}(-1)^i
	\operatorname{Tr}\!\left(
	\operatorname{Frob}_{k_n} 
	\mid 
	\mathcal{H}^i(K)_{\bar{x}}
	\right),
	\]
	where the geometric Frobenius~$\operatorname{Frob}_{k_n}$ of~$k_n$
	acts on the stalk~$K_{\bar{x}}$.
	
	We now specialize to $G = \mathbb{G}_{m,k}$.    
Following~\cite[Ch.~2]{katz_book} (see also~\cite{katz_p}), let \(\mathcal{P}\)
denote the full subcategory of \(\mathrm{Perv}(\mathbb{G}_{m,\bar{k}})\)
consisting of those perverse sheaves \(N\) such that, for every
\(M \in \mathrm{Perv}(\mathbb{G}_{m,\bar{k}})\), both convolutions
\(N *_! M\) and \(N *_* M\) remain perverse. 	The category~$\mathcal{P}$ is a neutral Tannakian category, 
where the tensor product is given by the \emph{middle convolution}
\(N *_\mathrm{mid} M := \operatorname{Im}\!\left(N *_! M \rightarrow N *_* M\right)\). 
 The irreducible objects in~$\mathcal{P}$ are precisely the irreducible perverse sheaves 
	in~$\mathrm{Perv}(\mathbb{G}_{m,\bar{k}})$ that lie in~$\mathcal{P}$.  
	
	Define \(\mathcal{P}_{\mathrm{arith}}\) to be the full subcategory of
	\(\mathrm{Perv}(\mathbb{G}_{m,k})\) consisting of those perverse sheaves \(N\)
	whose base change \(N_{\bar{k}}\) lies in \(\mathcal{P}\).   
	When the context is clear,  we use the same symbol \(N\) for both the object in
	\(\mathcal{P}_{\mathrm{arith}}\) and its base change \(N_{\bar{k}}\) on
	\(\mathbb{G}_{m,\bar{k}}\).

	For any $N \in \mathcal{P}_{\mathrm{arith}}$ (resp.\ $\mathcal{P}$), denote by \(\langle N \rangle_{\mathrm{arith}} \) (resp.\ \( \langle N \rangle_{\mathrm{geom}}\)) the subcategory tensor-generated by~$N$.
	With respect to a chosen fibre functor
	\[
	\omega : \langle N\rangle_{\mathrm{arith}} 
	\longrightarrow 
	\operatorname{Vect}_{\overline{\Q}_\ell},
	\]
	there exists an affine algebraic group 
	$G_{\mathrm{arith},N}$ over~$\overline{\Q}_\ell$
	representing the functor of tensor automorphisms
	$\underline{\operatorname{Aut}}^{\otimes}(\omega)$.  
	By Tannakian duality
	\cite[Thm.~2.11]{tannakian}, \(\omega\) induces a   tensor-equivalence
	\begin{equation}\label{tannakian_equivalence}
		\langle N\rangle_{\mathrm{arith}}
		\;\simeq\;
		\operatorname{Rep}^{\mathrm{fd}}_{\overline{\Q}_\ell}\!\left(G_{\mathrm{arith},N}\right),
	\end{equation}
	where 
	$\operatorname{Rep}^{\mathrm{fd}}_{\overline{\Q}_\ell}(G)$
	denotes the category of finite-dimensional 
	$\overline{\Q}_\ell$-representations of~$G$.
	
	Similarly, 
	there is an affine algebraic group 
	$G_{\mathrm{geom},N}$ over~$\overline{\Q}_\ell$
	such that
	\[
	\langle N\rangle_{\mathrm{geom}}
	\;\simeq\;
	\operatorname{Rep}^{\mathrm{fd}}_{\overline{\Q}_\ell}\!\left(G_{\mathrm{geom},N}\right).
	\]

Fix a field isomorphism
\(\tau : \overline{\mathbb{Q}}_\ell \xrightarrow{\sim} \mathbb{C}\),
and henceforth regard all objects over \(\overline{\mathbb{Q}}_\ell\)
as defined over \(\mathbb{C}\) via \(\tau\).

Finally, we note that the framework of Katz for \(\mathbb{G}_m\) has been
extended to arbitrary connected commutative algebraic groups over finite fields
by Forey, Fres\'an, and Kowalski~\cite{arithmetic}. In this paper we work over $\mathbb{G}_m$, but one may likewise work over the torus
$\mathbb{B}^\times =\operatorname{Res}_{B/k}\mathbb{G}_m$ by pushing forward along
$i_f:\A^1_k[1/f]\to \mathbb{B}^\times$.  For a related statement in the general setting of a connected commutative algebraic group $G$, see \cite[Prop.~4.3]{arithmetic}.

	\begin{defi}\label{def_good}
		Let $N$ be a perverse sheaf on~$\mathbb{G}_{m,k}$, and let $\rho$ be a character of~$k^\times$.  
		We say that $\rho$ is \emph{good for $N$} if, writing 
		$j:\mathbb{G}_{m,\bar{k}} \hookrightarrow \mathbb{P}^1_{\bar{k}}$ 
		for the open immersion, the natural “forget-supports'' morphism
		\[
		j_!\left(N \otimes \mathcal{L}_\rho\right)
		\;\longrightarrow\;
		Rj_*\left(N \otimes \mathcal{L}_\rho\right)
		\]
		is an isomorphism.
	\end{defi}
	
	\begin{remark}\label{remark}
	The notion of being \emph{good} is exactly that studied by Katz in~\cite[Ch.~4]{katz_book}.  
	For a character~\(\rho\) that is good for \(N\), the construction
		\[
		\omega : 
		M \;\longmapsto\; 
		H_c^0\!\left(\mathbb{G}_{m,\bar{k}},\, M \otimes \mathcal{L}_\rho\right)
		\]
	defines a fibre functor on the Tannakian subcategory
	\(\langle N \rangle_{\mathrm{arith}}\) (see~\cite[Thm.~4.1]{katz_book}).   
		If~$N$ is $\iota$-pure of weight~$0$ and arithmetically semisimple, then
		for any \(M \in \langle N \rangle_{\mathrm{arith}}\):	
		\begin{enumerate}
			\item 
			$M$ is $\iota$-pure of weight~$0$ and arithmetically semisimple 
		(see~\cite[Thm.~3.33]{arithmetic}, \cite[p.~25]{katz_book}).

			\item 
			The Frobenius action on
			$H_c^0\!\left(\mathbb{G}_{m,\bar{k}},\, M \otimes \mathcal{L}_\rho\right)$ 
			is pure of weight~$0$ (see~\cite[Thm.~4.1]{katz_book}).
			
			\item 
			The Tannakian dimension of~$M$ satisfies
			\[
			\dim M 
			\;=\;
			\dim H_c^0\!\left(\mathbb{G}_{m,\bar{k}},\, M \otimes \mathcal{L}_\rho\right)  
			\;=\;
			\chi_c\!\left(\mathbb{G}_{m,\bar{k}},\, M \otimes \mathcal{L}_\rho\right) 
			\;=\;
			\chi_c\!\left(\mathbb{G}_{m,\bar{k}},\, M\right),
			\]
		using the vanishing of
		\(H_c^i\!\left(\mathbb{G}_{m,\bar{k}},\, M \otimes \mathcal{L}_\rho\right)\) for \(i\neq 0\)
		(see~\cite[Lem.~2.1]{katz_book}).
		\end{enumerate}
	\end{remark}
	
Fix the fibre functor $\omega$ on $\langle N\rangle_{\mathrm{arith}}$ as in
Remark~\ref{remark}. Under~\eqref{tannakian_equivalence}, the object $N$
corresponds to the faithful representation $\omega(N)$ of $G_{\mathrm{arith},N}$, hence we may regard
	\[
	G_{\mathrm{geom},N} 
	\;\subset\;
	G_{\mathrm{arith},N}
	\;\subset\;
	\operatorname{GL}\left(  \omega(N)\right),
	\]
where $G_{\mathrm{geom},N}\subset G_{\mathrm{arith},N}$ is the closed immersion induced by
the base-change tensor functor
$\langle N\rangle_{\mathrm{arith}} \to \langle N\rangle_{\mathrm{geom}}$.
Moreover, every finite-dimensional representation of
$G_{\mathrm{geom},N}$ or $G_{\mathrm{arith},N}$ arises as a subquotient of a
  finite direct sum of tensor words in \(\omega(N)\) and
\(\omega(N)^{\vee}\); see~\cite[§4.14]{milne}.

For any $M\in \mathrm{Perv}(\mathbb{G}_{m,\bar k})$ there exists a dense open subset
$U\subset \mathbb{G}_{m,\bar k}$ such that $M|_U$ is a lisse sheaf $\mathcal F$.
We call $\operatorname{gen.rank}(M):=\operatorname{rank}(\mathcal F)$ the
\emph{generic rank} of $M$; this is well defined since $\mathbb{G}_{m,\bar k}$ is integral.

With these conventions in place, we recall the following construction from
\cite[Thm.~5.1]{katz}.
	
	\begin{theorem}[{\cite[Theorem~5.1]{katz}}]\label{theorem_sheaf_N}
		Let $f \in k[X]$ be a squarefree monic polynomial of degree $n \ge 2$, set 
		$B := k[X]/(f(X))$, and let $\chi$ be a character of $B^\times$.  
		Suppose the component  characters $\chi_i$ of $\chi$, as in~\eqref{character_component}, satisfy the following conditions:
		\begin{enumerate}
			\item The characters $\chi_i$ are pairwise distinct.
			\item The product $\prod_i \chi_i$ is nontrivial (equivalently, $\chi$ is nontrivial on $k^\times$).
			\item For at least one index $i$, one has $\chi_i^n \neq \prod_i \chi_i$.
		\end{enumerate}
		
		Fix $\lambda \in k^\times$.  
		Then  $\lambda f$ defines a $k$-morphism
		\(\lambda f : \mathbb{A}^1_k[1/f] \longrightarrow \mathbb{G}_{m,k},\)
		and we consider the perverse sheaf
		\[
		N(\lambda,\chi) := [\lambda f]_*\left(i_f^*\mathcal{L}_\chi\right)\!\left(\tfrac{1}{2}\right)[1]
		\]
		on $\mathbb{G}_{m,k}$.
		
		Then the following properties hold:
		\begin{enumerate}
			\item $N(\lambda,\chi)$ is geometrically irreducible, pure of weight~$0$, and lies in the 
			Tannakian category $\mathcal{P}_{\mathrm{arith}}$.  
			It has generic rank~$n$, Tannakian dimension~$n-1$,   and at most $2n$ bad characters.
			\item $N(\lambda,\chi)$ is geometrically Lie-irreducible in~$\mathcal{P}$.
			\item   $N(\lambda,\chi)$  has Tannakian groups 
			\[
			G_{\mathrm{geom},N(\lambda,\chi)} = G_{\mathrm{arith},N(\lambda,\chi)} = \operatorname{GL}(n-1).
			\]
		\end{enumerate}\qed
	\end{theorem}

	\begin{proposition}\label{prop_twisted_estimate}
		Let $\chi$ be a character of $B^\times$ whose component characters $\chi_i$ satisfy the 
		assumptions of Theorem~\ref{theorem_sheaf_N}.  
		Assume that $q := |k|$ satisfies $\sqrt{q} \ge   2n+1$.
		
		For each character $\rho$ of $k^\times$ that is \emph{good} for $N(\lambda,\chi)$ 
		in the sense of Definition~\ref{def_good},
		denote by
		\[
		\theta_{k,\lambda,f,\chi,\rho} \in U(n-1)
		\]
		the conjugacy class whose reversed characteristic polynomial is
		\[
		\det(1 - T\theta_{k,\lambda,f,\chi,\rho}) 
		= \det\!\left(1 - T\,\operatorname{Frob}_k \,\middle|\, 
		H_c^0\!\left(\mathbb{G}_{m, \bar{k}},\, N(\lambda,\chi)\otimes \mathcal{L}_\rho\right)\right).
		\]
		
		Let $\Lambda$ be a nontrivial irreducible representation of $U(n-1)$ contained in 
		$\operatorname{std}^{\otimes a} \otimes (\operatorname{std}^\vee)^{\otimes b}$, 
		and fix $c \in k^\times$.  
		Then the following estimate holds 
		\[	\left|
		\sum_{\rho \in \operatorname{Good}(k,\lambda,f,\chi)} 
		\rho(c)\,
		\operatorname{Tr}\!\left(\Lambda(\theta_{k,\lambda,f,\chi,\rho})\right)
		\right|
		\;\le\;
		|\operatorname{Good}(k,\lambda,f,\chi)|\cdot 
		\frac{(2 a+2b+1) n ^{a+b}}{\sqrt{q}}.\]
	\end{proposition}

	\begin{proof}
	This result is analogous to~\cite[Cor.~5.2]{katz}, and follows from parts of the
	proof of Theorem~7.2, Remark~7.5, and the proof of Theorem~28.1 in~\cite{katz_book}.
	In our application, we take \(N = N(\lambda,\chi)\), and in place of
	\(\operatorname{Tr}(\operatorname{Frob}_{k,1} \mid M)\) we obtain the shifted term
	\(\operatorname{Tr}(\operatorname{Frob}_{k,c^{-1}} \mid M)\).
	
	From Remark~\ref{remark}\,(2), the Frobenius action on
	\(H_c^0\!\left(\mathbb{G}_{m,\bar{k}},\, M \otimes \mathcal{L}_\rho\right)\)
	has unitary eigenvalues, giving a well-defined conjugacy class
	\(\theta_{k,\lambda,f,\chi,\rho} \in U(n{-}1)\).
	By Weyl's unitary trick (see~\cite[Prop.~III.8.6]{compact}), every irreducible
	representation \(\Lambda\) of \(U(n{-}1)\) is the restriction of an irreducible
  representation \(\Lambda_{\mathbb{C}}\) of \(\operatorname{GL}(n{-}1)\).

		Let $M$ denote the object corresponding to $\Lambda_{\mathbb{C}}$ 
	under the
	Tannakian equivalence
		\[
		\langle N(\lambda,\chi) \rangle_{\mathrm{arith}}
		\;\simeq\;
		\operatorname{Rep}^{\mathrm{fd}}\!\bigl(G_{\mathrm{arith},N(\lambda,\chi)}\bigr)
		\;\simeq\;
		\operatorname{Rep}^{\mathrm{fd}}\!\bigl(\operatorname{GL}(n{-}1)\bigr),
		\]
		where the final equivalence follows from Theorem~\ref{theorem_sheaf_N}.
		
		For any character~$\rho$ that is good for~$N(\lambda,\chi)$, we have  
		\begin{equation}\label{trace_equal_1}
			\operatorname{Tr}\! \left(\Lambda(\theta_{k,\lambda,f,\chi,\rho})\right)
			= 
			\operatorname{Tr}\!\left(
			\operatorname{Frob}_k 
			\,\middle|\,
			H_c^0\!\left(\mathbb{G}_{m,\bar{k}},\, M \otimes \mathcal{L}_\rho\right)
			\right).
		\end{equation}
By Remark~\ref{remark}\,(1),  any \(K \in \langle N(\lambda,\chi)\rangle_{\mathrm{arith}}\) is semisimple.  
Since tensoring with a Kummer sheaf \(\mathcal{L}_\rho\) is an autoequivalence of
\(\mathcal{P}\), the object \(K \otimes \mathcal{L}_\rho\) is again semisimple.
Therefore, applying \cite[Lem.~2.1]{katz}, for any \(K \in \langle N(\lambda,\chi)\rangle_{\mathrm{arith}}\), 
\[
H_c^i\!\left(\mathbb{G}_{m,\bar{k}},\, K \otimes \mathcal{L}_\rho\right)
= 0
\qquad\text{for all } i \neq 0.
\]

Applying the Grothendieck--Lefschetz trace formula \cite[Exp.~III, §4]{trace} gives
\[
\sum_{i\in\mathbb{Z}}(-1)^i\,
\operatorname{Tr}\!\left(\operatorname{Frob}_k \,\middle|\, H_c^i\!\left(\mathbb{G}_{m,\bar{k}},\,M \otimes \mathcal{L}_{\rho}\right)\right)
=
\sum_{s \in \mathbb{G}_m(k)} 
\rho(s)\,
\operatorname{Tr}\!\left(
\operatorname{Frob}_{k,s} \mid M
\right).
\]
Using the vanishing above, the left-hand side reduces to the $i=0$ term, yielding
	\begin{equation}\label{trace}
	\operatorname{Tr}\!\left(
	\operatorname{Frob}_k 
	\,\middle|\, 
	H_c^0(\mathbb{G}_{m,\bar{k}}, M \otimes \mathcal{L}_\rho)
	\right)
	=
	\sum_{s \in \mathbb{G}_m(k)} 
	\rho(s)\,
	\operatorname{Tr}\!\left(
	\operatorname{Frob}_{k,s} \mid M
	\right).
\end{equation}

		By~\cite[II,~3.3.1]{delign}, the cohomology group
		$H_c^0\!\left(\mathbb{G}_{m,\bar{k}},\, M \otimes \mathcal{L}_\rho\right)$
		is mixed of weights~$\le 0$.  
		Moreover, by Remark~\ref{remark}\,(2) and~(3), we have
		\[
		\left|
		\operatorname{Tr}\!\left(
		\operatorname{Frob}_k 
		\mid 
		H_c^0\!\left(\mathbb{G}_{m,\bar{k}},\, M \otimes \mathcal{L}_\rho\right)
		\right)
		\right|
		\;\le\;
		\dim M.
		\]

		Let $\widehat{k^\times}$ denote the group of all multiplicative characters of~$k^\times$.  
		Write 
		\[
		\operatorname{Good} = \operatorname{Good}(k,\lambda,f,\chi)
		\quad\text{and}\quad
		\operatorname{Bad} = \widehat{k^\times} \setminus \operatorname{Good}.
		\]
		Then
		\begin{align*}
			\frac{1}{|\operatorname{Good}|}
			\sum_{\rho \in \operatorname{Good}} 
			\rho(c)\,
			\operatorname{Tr}\!\left(\Lambda(\theta_{k,\lambda,f,\chi,\rho})\right)
			&=\;
			\frac{1}{|\widehat{k^\times}|}
			\sum_{\rho \in \widehat{k^\times}}
			\rho(c)\,
			\operatorname{Tr}\!\left(\Lambda(\theta_{k,\lambda,f,\chi,\rho})\right) \\[6pt]
			&\quad+\;
			\left(
			\frac{1}{|\operatorname{Good}|} - \frac{1}{ |\widehat{k^\times} |}
			\right)
			\sum_{\rho \in \widehat{k^\times}}
			\rho(c)\,
			\operatorname{Tr}\!\left(\Lambda(\theta_{k,\lambda,f,\chi,\rho})\right) \\[6pt]
			&\quad-\;
			\frac{1}{|\operatorname{Good}|}
			\sum_{\rho \in \operatorname{Bad}} 
			\rho(c)\,
			\operatorname{Tr}\!\left(\Lambda(\theta_{k,\lambda,f,\chi,\rho})\right).
		\end{align*}
		By \eqref{trace_equal_1}, \eqref{trace}, and orthogonality of characters on $k^\times$,
		the first term equals 
		\[
		\operatorname{Tr}\!\left(\operatorname{Frob}_{k, c^{-1}} \mid M\right),
		\]
		the second term equals 
		\[
		\frac{|\operatorname{Bad}|}{|\operatorname{Good}|}
		\operatorname{Tr}\!\left(\operatorname{Frob}_{k, c^{-1}} \mid M\right),
		\]
		and the third term is bounded by 
		\[
		\frac{|\operatorname{Bad}|}{|\operatorname{Good}|} \, \dim M.
		\]
		
	If \(M\) were punctual, then it would be of the form
	\(\alpha^{\deg} \otimes \delta_{t_0}\) for some unitary scalar \(\alpha\), where
	\(\delta_{t_0}\) denotes the skyscraper sheaf supported at a point
	\(t_0 \in \mathbb{G}_m(k)\).  By~\cite[Cor.~6.6]{katz_book}, since
	\(G_{\mathrm{geom},N(\lambda,\chi)} = \operatorname{GL}(n{-}1)\) is connected, we
	must have \(t_0 = 1\).  But in that case \(M\) would be geometrically trivial,
	contradicting the equality
	\(G_{\mathrm{geom},N(\lambda,\chi)} = G_{\mathrm{arith},N(\lambda,\chi)}\).
	Therefore \(M\) is non-punctual.

		Hence, there exist a closed immersion \(i\colon Y \hookrightarrow \mathbb{G}_m\), a dense open immersion
		\(j\colon U \hookrightarrow Y\)
		and a  lisse $\overline{\mathbb{Q}}_\ell$-sheaf 
		$\mathcal{F}$ on~$U$ such that
		\[
		M = i_* j_{!*} \mathcal{F}[1].
		\]
		Since~$M$ is pure of weight~$0$ by Remark~\ref{remark}\,(1), 
		and both~$i_*$ and~$j_{!*}$ preserve weights, it follows that \(\mathcal{F}\) is pure of
		weight \(-1\).  Write \(\operatorname{gen.rank}(M)\) for the generic rank of
		\(M\), i.e.\ the rank of \(\mathcal{F}\). 
		Then,   we have
		\[
		\left|
		\operatorname{Tr}\!\left(
		\operatorname{Frob}_{k, c^{-1}} \mid M
		\right)
		\right|
		=
		\left| -
		\operatorname{Tr}\!\left(
		\operatorname{Frob}_{k, c^{-1}} \mid \mathcal{F}
		\right)
		\right|
		\;\le\;
		\frac{\operatorname{gen.rank}(M)}{\sqrt{q}}.
		\]	The final inequality follows from Deligne's Riemann Hypothesis applied to the
		pure weight \(-1\) lisse sheaf \(\mathcal{F}\).
		
		Since $\Lambda$ occurs in 
		$\operatorname{std}^{\otimes a} \otimes (\operatorname{std}^\vee)^{\otimes b}$, 
		the object $M$ is a direct summand  of 
		\[
		N(\lambda,\chi)^{\otimes a} \otimes (N(\lambda,\chi)^{\vee })^{\otimes b}.
		\]
		From Theorem~\ref{theorem_sheaf_N}, we know that 
		$N(\lambda,\chi)$ has generic rank~$n$, 
		Tannakian dimension~$n-1$, and at most~$2n$ bad characters.  
		Hence,
		\[
		\dim M 
		\le 
		\dim \!\left(
		N(\lambda,\chi)^{\otimes a} \otimes (N(\lambda,\chi)^{\vee })^{\otimes b}
		\right)
		=\left(\dim N(\lambda,\chi)\right)^{a+b}= (n-1)^{a+b}.
		\]
		Moreover,
		\begin{align*}
			\operatorname{gen.rank}(M)
			&\le 
			\operatorname{gen.rank}\!\left(
			N(\lambda,\chi)^{\otimes a} 
			\otimes 
			\left(N(\lambda,\chi)^{\vee}\right)^{\otimes b}
			\right) \\[4pt]
			&\le 
			(a+b)\,\left(\dim N(\lambda,\chi)\right)^{a+b-1}\,
			\operatorname{gen.rank}\left(N(\lambda,\chi)\right) \\[4pt]
			&= (a+b)(n-1)^{a+b-1} n,
		\end{align*}
		where the second inequality follows from 
		\cite[p.~175]{katz_book}.

		If $\sqrt{q} \ge 2n + 1$, then 
		$|\operatorname{Bad}| \le \sqrt{q} - 1$.  
		Therefore,
		\begin{align*}
			\frac{1}{|\operatorname{Good}|}
			\sum_{\rho \in \operatorname{Good}} 
			\rho(c)\,
			\operatorname{Tr}\! \left(\Lambda(\theta_{k,\lambda,f,\chi,\rho})\right)
			&\le 
			\frac{\operatorname{gen.rank}(M)}{\sqrt{q}}
			+ \frac{1}{\sqrt{q}} \cdot \frac{\operatorname{gen.rank}(M)}{\sqrt{q}}
			+ \frac{1}{\sqrt{q}} \dim M  \\[4pt]
			&\le 
			\left(1+\frac{1}{\sqrt{q}}\right)
			\frac{(a+b)\,(n-1)^{a+b-1}\,n}{\sqrt{q}}
			+ \frac{(n-1)^{a+b}}{\sqrt{q}} \\[4pt]
			&\le 
			\frac{(2a + 2b + 1)\, n^{a+b}}{\sqrt{q}}.
		\end{align*}
	\end{proof}
	
	\begin{defi}
		Let $\chi$ be a character of $B^\times$.  
		We say that $\chi$ is \emph{totally ramified   generic} if 
		\begin{enumerate}
			\item $\chi$ is totally ramified in the sense of Definition~\ref{def_TotRam}, and 
			\item the component characters $\chi_i$ of $\chi$ satisfy the three conditions of 
			Theorem~\ref{theorem_sheaf_N}.
		\end{enumerate}
		We denote by $\operatorname{TotRamGen}(k,f)$ the set of totally ramified  generic characters of $B^\times$.
	\end{defi}

	\begin{proposition}\label{prop_sum_totRam}
		Suppose $\sqrt{q} \ge  2n+1$.  
		Let $\Lambda$ be a nontrivial irreducible representation of $U(n-1)$ that occurs in 
		$\operatorname{std}^{\otimes a} \otimes (\operatorname{std}^\vee)^{\otimes b}$, 
		and let $c \in B^\times$.  
		Then the following estimate holds
		\[
		\left| 
		\sum_{\chi \in \operatorname{TotRamGen}(k,f)} 
		\chi(c)\,\operatorname{Tr}\!\left(\Lambda(\theta_{k,f,\chi})\right)
		\right|
		\;\le\;
		\left|\operatorname{TotRamGen}(k,f)\right| \cdot 
		\frac{(2a + 2b + 1)\, n^{a+b}}{\sqrt{q}}.
		\]
		
	\end{proposition}

	\begin{proof}
		By~\cite[Lem.~2.1]{katz}, we have
		\[
		L(\chi, T)
		=\exp\!\left(
		\sum_{r \ge 1} 
		\frac{S_r\, T^r}{r}
		\right),
		\]
		where
		\begin{align*}
			S_r 
			&=\!
			\sum_{t \in \mathbb{A}^1[1/f](k_r)}
			\chi\!\left(\operatorname{Norm}_{k_r/k}(X - t)\right)\\
			&=\!
			\sum_{t \in \mathbb{A}^1[1/f](k_r)}
			\operatorname{Tr}\!\left(
			\operatorname{Frob}_{k_r,t}
			\mid 
			i_f^*\mathcal{L}_\chi
			\right)
			\\
			&= 
			\sum_{i = 1,2}
			(-1)^i 
			\operatorname{Tr}\!\left(
			\operatorname{Frob}_k^r
			\mid 
			H_c^i\!\left(
			\mathbb{A}^1[1/f]_{\bar{k}},\, i_f^*\mathcal{L}_\chi
			\right)
			\right).
		\end{align*}
		Therefore,
		\[
		L(\chi, T)
		=\;
		\frac{
			\det\!\left(
			1 - T\,\operatorname{Frob}_k 
			\mid H_c^1\!\left(\mathbb{A}^1[1/f]_{\bar{k}},\, i_f^*\mathcal{L}_\chi\right)
			\right)
		}{
			\det\!\left(
			1 - T\,\operatorname{Frob}_k 
			\mid H_c^2\!\left(\mathbb{A}^1[1/f]_{\bar{k}},\, i_f^*\mathcal{L}_\chi\right)
			\right)
		}.
		\]
		For 
		$\chi \in \operatorname{TotRamGen}(k,f)$, 
		by~\cite[Lem.~3.2]{katz} we have 
		\(H_c^2\!\left(\mathbb{A}^1[1/f]_{\bar{k}},\, i_f^*\mathcal{L}_\chi\right) = 0,\)
		and thus
		\[
		L(\chi, T)
		=\;
		\det\!\left(
		1 - T\,\operatorname{Frob}_k 
		\mid H_c^1\!\left(\mathbb{A}^1[1/f]_{\bar{k}},\, i_f^*\mathcal{L}_\chi\right)
		\right).
		\]

		Moreover, since the morphism 
		$[\lambda f]$ is finite, the higher direct images 
		$R^i[\lambda f]_*$ vanish for $i > 0$.  Hence the Leray spectral sequence
		degenerates and we obtain
		\begin{align*}
			H_c^0\!\left(\mathbb{G}_{m,\bar{k}},\, N(\lambda,\chi) \otimes \mathcal{L}_\rho\right)
			&\cong 
			H_c^0\!\left(\mathbb{A}^1[1/f]_{\bar{k}},\, 
			i_f^*\mathcal{L}_\chi \otimes [\lambda f]^*\mathcal{L}_\rho(1/2)[1]\right) \\[4pt]
			&\cong 
			H_c^1\!\left(\mathbb{A}^1[1/f]_{\bar{k}},\, 
			i_f^*\mathcal{L}_\chi \otimes [\lambda f]^*\mathcal{L}_\rho(1/2)\right).
		\end{align*}
		Here, the half Tate twist $(1/2)$ multiplies the trace of Frobenius by 
		$|k|^{-1/2} = q^{-1/2}$.
		
		By~\eqref{norm_relation}, we have 
		$\operatorname{Norm}_{B/k} \circ i_f = (-1)^n f$.  
		Denote by $\rho_{\operatorname{Norm}}$ the character of $B^\times$ defined by 
		$\rho_{\operatorname{Norm}} := \rho \circ \operatorname{Norm}_{B/k}$.  
		Then 
		\[
		[(-1)^n f]^*\mathcal{L}_\rho 
		= i_f^*\!\left(\operatorname{Norm}_{B/k}^*\mathcal{L}_\rho\right)
		= i_f^*\mathcal{L}_{\rho_{\operatorname{Norm}}}.
		\]
		Therefore, the conjugacy class 
		$\theta_{k,(-1)^n f,\chi,\rho}$, as defined in Proposition~\ref{prop_twisted_estimate}, 
		coincides with the conjugacy class 
		$\theta_{k,f,\chi\rho_{\operatorname{Norm}}}$ 
		as defined in~\eqref{def_conjugacy_class}.

		By \cite[Lem.~6.1]{katz}, two totally ramified generic characters 
		$\chi, \chi' \in B^\times$ are said to be \emph{equivalent} if
		\[
		\chi' = \chi \,\rho_{\operatorname{Norm}}
		\]
		for some (necessarily unique) character $\rho\in \operatorname{Good}(k,(-1)^n f,\chi)$.  
		This defines an equivalence relation on 
		$\operatorname{TotRamGen}(k,f)$.		
	Partitioning $\operatorname{TotRamGen}(k,f)$ into equivalence classes and applying Proposition~\ref{prop_twisted_estimate}, we obtain
		\[
		\begin{aligned}
			&\left| 
			\sum_{\chi \in \operatorname{TotRamGen}(k,f)} 
			\chi(c)\,\operatorname{Tr}\!\left(\Lambda\!\left(\theta_{k,f,\chi}\right)\right) 
			\right| \\[1ex]
			&\qquad= 
			\left| 
			\sum_{[\chi] \in \operatorname{TotRamGen}(k,f)/\!\sim }
			\chi(c) 
			\sum_{\rho \in \operatorname{Good} \left(k,(-1)^n f,\chi\right)} 
			\rho\!\left(\operatorname{Norm}_{B/k}(c)\right)
			\operatorname{Tr}\!\left(\Lambda\!\left(\theta_{k,f,\chi\rho_{\operatorname{Norm}}}\right)\right) 
			\right| \\[1ex]
			&\qquad= 
			\left| 
			\sum_{[\chi] \in \operatorname{TotRamGen}(k,f)/\!\sim }
			\chi(c) 
			\sum_{\rho \in \operatorname{Good} \left(k,(-1)^n f,\chi\right)} 
			\rho\!\left(\operatorname{Norm}_{B/k}(c)\right)
			\operatorname{Tr}\!\left(\Lambda\!\left(\theta_{k,(-1)^n f,\chi,\rho}\right)\right) 
			\right| \\[1ex]
			&\qquad\leq 
			\left|\operatorname{TotRamGen}(k,f)\right| \cdot 
			\frac{(2a + 2b + 1)\, n^{a+b}}{\sqrt{q}}.
		\end{aligned}
		\]
		
	\end{proof}

	\begin{corollary}\label{prop_trace_bound_2}
		Suppose that
		\(\sqrt{q} \ge 2n + 1.\)
		Then for any \(a \in B^\times\) and integers \( i\geq j\ge 1 \), we have
		\[
		\left|
		\sum_{\chi \in \operatorname{TotRam}(k,f)} 
		\chi(a^{-1})\,
		\operatorname{Tr}\!\left(\theta_{k,f,\chi}^i\right)
		\operatorname{Tr}\!\left(\theta_{k,f,\chi}^j\right)
		\right|
		\;\le\;
		p(i{+}j)\,\chi_{i,j}(1)\,
		\left|\operatorname{TotRam}(k,f)\right|\,
		\frac{2(i{+}j{+}1)\,n^{\,i+j}}{\sqrt{q}},
		\]
		where \(p(m)\) denotes the number of partitions of~\(m\), and 
		\(\chi_{i,j}(1)\) is the dimension of the irreducible representation of~\(S_{i+j}\) 
		associated with the partition~\((i,j)\).
	\end{corollary}

	\begin{proof}
		We decompose the sum into two parts: the contribution from 
		$\chi \in \operatorname{TotRamGen}(k,f)$ and that from 
		$\chi \in \operatorname{TotRam}(k,f) \setminus \operatorname{TotRamGen}(k,f)$.
		
		We first consider the contribution from the former set.  
		Recall the power-sum symmetric functions
		\[
		P_j(x_1, x_2, \dots, x_n) \;=\; \sum_{i=1}^n x_i^j,
		\]
		and for a partition $\lambda = 1^{a_1} 2^{a_2} \cdots k^{a_k}$, define
		\[
		P_\lambda \;=\; \prod_{j=1}^k P_j^{\,a_j}.
		\]
		Here $\lambda$ is a partition of $K = a_1 + 2a_2 + \cdots + ka_k$.  
		As $\lambda$ ranges over all partitions of $K$, the functions $P_\lambda$ form a basis 
		for the space of homogeneous symmetric polynomials of degree~$K$ in $n$ variables 
		for all $n \ge K$.
		
		A second basis for this space is given by the Schur functions $s_\mu$.  
		They are related to the power sums by the Frobenius character formula
		\[
		P_\lambda \;=\; \sum_{\mu \vdash K} \chi_\lambda(\mu)\, s_\mu,
		\]
		where $\chi_\lambda(\mu)$ denotes the value of the irreducible character of 
		  $S_K$ associated to $\lambda$, 
		evaluated on the conjugacy class of cycle type~$\mu$.
		
		In our setting,  
		$\operatorname{Tr}\!  \left(\theta_{k,f,\chi}^i\right) 
		\operatorname{Tr} \left(\theta_{k,f,\chi}^j\right)$ 
	is the evaluation of $P_{i,j}$ 
		on the eigenvalues of~$\theta_{k,f,\chi}$.  
		Thus, for $\lambda = (i,j)$ we obtain
		\[
		P_{i,j} \;=\; \sum_{\mu \vdash i+j} \chi_{i,j}(\mu)\, s_\mu.
		\]
	Let $\mathbb{S}_\mu(\C^{\,n-1})$ denote the Weyl module associated with $\mu$ 
	(see~\cite[p.~26]{fulton}).  
	It is irreducible for $\operatorname{GL}_{n-1}(\C)$ (and hence for $U_{n-1}(\C)$), 
	with character given by the Schur function $s_\mu$; see~\cite[Thm.~6.3]{fulton}.
	Applying Proposition~\ref{prop_sum_totRam} to each $\mathbb{S}_\mu(\C^{\,n-1})$ and using the trivial bound $|\chi_\lambda(\mu)| \le \chi_\lambda(1)$, we obtain, for any $a \in B^\times$,
		\[
		\left|
		\sum_{\chi \in \operatorname{TotRamGen}(k,f)} 
		\chi(a^{-1}) \,
		\operatorname{Tr}\!\left(\theta_{k,f,\chi}^i\right)
		\operatorname{Tr}\!\left(\theta_{k,f,\chi}^j\right)
		\right|
		\;\leq\;
		p(i{+}j)\,\chi_{i,j}(1)\,
		\left|\operatorname{TotRamGen}(k,f)\right|
		\cdot 
		\frac{(2i + 2j + 1)\, n^{\,i+j}}{\sqrt{q}}.
		\]
		where $p(m)$ denotes the number of partitions of $m$.

		For the second part, observe that it contains at most 
		$\left|\operatorname{TotRam}(k,f)\right| - \left|\operatorname{TotRamGen}(k,f)\right|$ terms.  
	Since $\theta_{k,f,\chi}  $ is a unitary conjugacy class, we have 
	$\left|\operatorname{Tr}(\theta_{k,f,\chi}^i)\operatorname{Tr}(\theta_{k,f,\chi}^j)\right|\le (n-1)^2$.
		
	As shown in~\cite[Lem.~6.5--6.6]{katz}, when \(q \ge n+1\) we have
		\[
		\begin{aligned}
			\left|\operatorname{TotRam}(k,f)\right| 
			- \left|\operatorname{TotRamGen}(k,f)\right|
			&\le
			P_{\mathrm{TRG},n}(q) - q^n + 1, \\[4pt]
			\left|\operatorname{TotRamGen}(k,f)\right|
			&\ge
			P_{\mathrm{TRG},n}(q),
		\end{aligned}
		\]
		where \(P_{\mathrm{TRG},n}(q)\) denotes the polynomial
		\[
		P_{\mathrm{TRG},n}(q)
		= (q - 1 - n)^n + (q - 2)^n - q^n + 1 - n \sum_{0 \le i \le n-1} q^i.
		\]
		Since \(P_{\mathrm{TRG},n}(q) - q^n + 1 \le \dfrac{P_{\mathrm{TRG},n}(q)}{\sqrt{q}}\),
		it follows that, for \(q \ge n+1\),
		\[
		\left|\operatorname{TotRam}(k,f)\right|
		- \left|\operatorname{TotRamGen}(k,f)\right|
		\;\le\;
		\frac{\left|\operatorname{TotRamGen}(k,f)\right|}{\sqrt{q}}
		\;\le\;
		\frac{\left|\operatorname{TotRam}(k,f)\right|}{\sqrt{q}}.
		\]
		Moreover,   the condition \(q \ge n+1\) is implied by \(\sqrt{q} \ge 2n + 1\). Combining the two estimates yields the desired bound.  
	\end{proof}

	We now apply the above estimates to bound $R(a,n)$ in~\eqref{R(a,n)}, thereby obtaining a saving of order $q^{-1/2}$.
	Combining~\eqref{R(a,n)} with Corollary~\ref{prop_trace_bound_2},  we have
	\[	|R(a,n)|
	\le 
	\sum_{1\leq i,j \le n}
	q^{\tfrac{i+j}{2}}\,
	p(i{+}j)\,\chi_{i,j}(1)\,
	\frac{2(i{+}j{+}1)\,n^{\,i+j}}{\sqrt{q}}
	+ 
	2(n{+}1)\,n^4\,q^{\,n-1},\]
where \(\chi_{i,j}(1)\) is understood to mean \(\chi_{\max(i,j),\,\min(i,j)}(1)\).
	
	By~\cite[(4.11)]{fulton},   
	\[
	\chi_{i,j}(1)
	= \frac{(i+j)!}{(i+1)!\,j!}\,\left(|i{-}j|{+}1\right),
	\]
	and hence
	\[
	(i{+}j{+}1)\,\chi_{i,j}(1)
	= 
	\binom{i{+}j{+}1}{j}\,\left(|i{-}j|{+}1\right)
	\;\le\;
	2^{\,i{+}j{+}1}\cdot 2n,
	\]
	where we used   \(i{+}j \le 2n\).
Together with the classical bound
\(p(m) \le e^{\,\pi\sqrt{2m/3}}\) \cite[Thm.~14.5]{apostol},
this yields
	\begin{align*}
		\sum_{1\leq i,j \le n}
		q^{\tfrac{i+j}{2}}\,
		p(i{+}j)\,\chi_{i,j}(1)\,
		\frac{2(i{+}j{+}1)\,n^{\,i+j}}{\sqrt{q}}
		&\le
		q^{\,n-\tfrac12}
		\sum_{1\leq i,j \le n}
		4\,(2n)^{\,i+j+1}\,
		e^{\,\pi\sqrt{2(i{+}j)/3}} \\[4pt]
		&\le
		q^{\,n-\tfrac12}\,
		(2n)^{\,2n+3}\,
		e^{\,\pi\sqrt{4n/3}}.
	\end{align*}
Set
	\[
	C(n)
	:= 
	(2n)^{\,2n+3}\,
	e^{\,\pi\sqrt{4n/3}}.
	\]
	Therefore, for all \(a \in B^\times\) and \(\sqrt{q} \ge 2n + 1\),  
	\[
	|R(a,n)|
	\;\le\;
	C(n)\,q^{\,n-\tfrac12}
	\;+\;
	2(n{+}1)\,n^4\,q^{\,n-1}.
	\]
	
	\smallskip
	From~\eqref{lower_bound_M}, we recall that
	\[
	M(a,n) \ge q^n - 2n^2 + \frac{n^4}{q^n}.
	\]
	To ensure that \( |R(a,n)| < M(a,n) \), it is enough to require
	\[
	q^n \ge 2\,C(n)\,q^{\,n-\tfrac12},
	\]
	which is equivalent to \( q^{1/2} \ge 2C(n) \).
	Thus we may take
	\[
	Q(n) = (2C(n))^2.
	\]
	Hence, whenever \( q \ge Q(n) \), we have \( |R(a,n)| < M(a,n) \) for every
	\(a \in B^\times\).
	Alternatively, one may take \( Q(n) \) in the simpler  form
	\[
	Q(n) = n^{\,\kappa n}
	\]
	for some absolute constant \( \kappa > 0 \); for instance, \( \kappa = 23 \) suffices.

This completes the proof of Theorem~\ref{thm:main}.  

	\smallskip

Finally, the same argument shows that in the definition of \(E(k,f)\) it is enough to
fix \(\deg f_1 = r\) and \(\deg f_2 = s\) with \(r+s = 2n\), rather than requiring
\(\deg f_1, \deg f_2 \le n\).  In the bound for \(R(a,n)\) above, this simply amounts to retaining a single
pair \((i,j) = (r,s)\) in place of summing over all \(i,j \le n\), and the resulting
estimate for a single pair \((i,j)\) with
\(i+j = 2n\) remains unchanged.

Fix an integer \(1 \le r < 2n\), and define
\[
E'(k,f)
:=
\left\{
a \in B^\times :
a \equiv f_1 f_2 \pmod{f},\ 
f_1,f_2 \in k[X]\ \text{monic irreducible},\ 
\deg f_1 = r,\ \deg f_2 = 2n-r
\right\}.
\]
Then for any integer \(n \ge 2\), any finite field \(k = \F_q\) with \(q \ge Q(n)\), 
and any squarefree polynomial \(f \in k[X]\) of degree \(n\), we likewise have
\[
E'(k,f) = B^\times.
\]

\smallskip

	\smallskip
\section*{Acknowledgments}
We thank Will Sawin for informing us of his work~\cite{sawin}, which proves a stronger and more general result, after an earlier version of this manuscript appeared on the arXiv.  
We are grateful to Nicholas M.~Katz for comments that clarified the presentation of Theorem~\ref{thm:main} and the constant \(Q(n)\).  
 We also thank
Sampa Dey for helpful discussions, and Pieter Moree for reading the manuscript
and offering suggestions that improved the exposition.

	\bibliographystyle{alpha}

	\bibliographystyle{plain}

\end{document}